\documentclass[a4paper,12pt]{article}
\usepackage{amsfonts}
\usepackage{amssymb,amsmath}
\usepackage{amsthm}
\usepackage{epsfig}                  
\usepackage{epstopdf}                
\usepackage{graphicx}
\usepackage{multirow}
\usepackage{mathrsfs}
\usepackage{comment}
\usepackage{bm}
\usepackage{booktabs}
\usepackage{enumitem}
\usepackage{listings}
\usepackage{color}
\usepackage[title]{appendix}
\usepackage{float}
\usepackage{algorithm}
\usepackage{algpseudocode}
\usepackage{fancyhdr}
\usepackage{setspace}
\usepackage{cases} 
\usepackage{chngcntr}
\usepackage{graphicx}
\usepackage{subcaption}
\usepackage{rotating}

\counterwithin{algorithm}{section} 

\usepackage{caption}
\DeclareMathOperator*{\argmin}{arg\,min}

\DeclareMathOperator*{\ri}{ri}
\DeclareMathOperator*{\Int}{int}

\newtheorem{theorem}{Theorem}[section]
\newtheorem{definition}{Definition}[section]
\newtheorem{lemma}{Lemma}[section]

\newtheorem{remark}{Remark}[section]

\newtheorem{example}{Example}[section]

\begin{document}

\title{\textbf{Adaptive relaxed CQ algorithm with dynamic stepsizes for the split feasibility problem}}

\author{
{\sc Qiao-Li Dong{\thanks{Corresponding author. email: dongql@lsec.cc.ac.cn}}}\\
{\footnotesize College of Science, Civil Aviation University of China, Tianjin 300300, China}\\
}
\date{}

\maketitle

\begin{abstract}
In this paper, we propose a relaxed CQ algorithm with dynamic stepsizes for the split feasibility problem (SFP), where the relaxed parameters are chosen adaptively. We establish the weak convergence of the proposed algorithm under standard assumptions and further derive its linear convergence  under the bounded linear regularity property of the SFP. Numerical results on LASSO problem demonstrate that the adaptively over-relaxed CQ algorithm with dynamic stepsizes outperforms the dynamic CQ algorithm  proposed in \cite{Xu-dynamic-CQ}.
\end{abstract}



\section{Introduction }

In this article, we consider the split feasibility problem (SFP), introduced by Censor and Elfving \cite{Censor1},
which is formulated as finding a point $x^*$ with the property
\begin{equation}
 \label{SFP}
x^*\in C\quad\hbox{and}\quad Ax^*\in Q,
\end{equation}
where $C$ and $Q$ are the nonempty closed convex subsets of the
 real Hilbert spaces $\mathcal {H}_1$ and
$\mathcal {H}_2$, respectively, and  $A :\mathcal {H}_1\rightarrow\mathcal {H}_2$  is a bounded
linear operator.
The SFP \eqref{SFP}  can be a model
for many inverse problems where constraints are imposed on the solutions in
the domain of a linear operator as well as in the operator's range (see, e.g. \cite{Dong-SFP-Nash-equilibrium,WangHu}).
Throughout this article, we assume that SFP (\ref{SFP}) is consistent, i.e., its solution set,  denoted by
$
\Omega =\{x\,|\,x\in C\,\hbox{and}\, Ax\in Q\},
$
 is nonempty.
 
The SFP can be reformulated as an unconstrainted optimization problem
\begin{equation}
 \label{MinP}
\min_{x\in\mathcal{H}} f(x)+\iota_C(x),
\end{equation}
where
$
f(x)=\frac12\|(I-P_Q)Ax\|^2
$
and $\iota_C$ is the indicator of $C.$
Recall that the objective function $f$ is convex and differentiable, with gradient
$
\nabla f(x)=A^*(I-P_Q)Ax.
$
Moreover, $\nabla f$ is Lipschitz continuous with Lipschitz constant $\|A\|^2$.
There exist numerous iterative methods for solving \eqref{SFP}, including the CQ algorithms \cite{{Byrne1},{Byrne0}}, general splitting methods \cite{Dong-general-splitting-method}, projection and contraction methods \cite{Dong-Optimal}, primal-dual-based iterative methods \cite{Zhaojing} and Polyak's gradient method \cite{Wangfh}, as well as their inertial variants \cite{Shehu-alternated-inertial-CQ,Shehu-inertial-CQ}. 
One of the simplest methods for solving the problem \eqref{SFP}, or equivalently problem \eqref{MinP}, is the gradient-projection method, also known as the forward-backward splitting method, which is given by
\begin{equation}
 \label{CQ}
x_{n+1}=P_C(x_n-\tau_n A^*(I-P_Q)A(x_n)).
\end{equation}
This method
was  referred as CQ algorithm since it involves the (metric) projections onto the sets $C$ and $Q$. 

There are three common strategies for choosing the stepsize $\tau_n$ in the CQ algorithm \eqref{CQ}. The first is to use a fixed stepsize $\tau_n\in (0,\frac{\displaystyle 2}{\displaystyle\|A\|^2})$ which depends on the the operator (matrix) norm $\|A\|$ (see \cite{{Byrne1},{Byrne0}}). Therefore, to implement the algorithm, one must first compute, or at least estimate, the operator norm of $A$, which is generally not an easy task in practice. The second strategy is to choose the stepsize $\tau_n$ adaptively by means of an Armijo-type line search (see \cite{Yang}). Determining $\tau_n$ in this way generally requires several evaluations of $\nabla f$ at each iteration, which can be computationally expensive because each evaluation involves a projection onto $Q$ and matrix-vector products. The third strategy is to use the dynamic stepsize introduced by L\'{o}pez et al. \cite{Xu-dynamic-CQ}, defined by
\begin{equation}
\label{dynamic_stepsize}
\tau_n=\rho_n\frac{f(x_n)}{\|\nabla f(x_n)\|^2},
\end{equation}
where $\rho_n\in(0,4).$ Note that the dynamic stepsize $\lambda_n$  does not need any prior information about the operator norm.  Since its inception in 2012, the dynamic stepsize has received much attention
due to its favorable numerical performance.

In this paper, we consider the relaxed variant of the CQ algorithm \eqref{CQ} as follows
\begin{equation}
\label{relaxed-CQ}
x_{n+1}=(1-\lambda_n)x_n+\lambda_nP_C\big(x_n - \tau_n A^*(I - P_Q)A x_n\big),
\end{equation}
where  $\lambda_n>0$ is an relaxed parameter.
Note that if $\lambda_n<1$,
then the scheme \eqref{relaxed-CQ} is an \texttt{underrelaxation} of \eqref{CQ} and, if $\lambda_n>1$, then the scheme \eqref{relaxed-CQ} is an \texttt{overrelaxation} of \eqref{CQ}. In particular,  $\lambda_n=1$ reduces \eqref{relaxed-CQ} to  \eqref{CQ}.

We propose an adaptive strategy for selecting $\lambda_n$ and show that our algorithm \eqref{relaxed-CQ} with dynamic stepsize \eqref{dynamic_stepsize} has a better contraction property than the  dynamic CQ algorithm \eqref{CQ}-\eqref{dynamic_stepsize} under certain conditions. Furthermore, we investigate the weak and linear convergence properties of the proposed method.  Numerical experiments on LASSO problem demonstrate the superior performance of the scheme \eqref{relaxed-CQ} with  dynamic stepsize \eqref{dynamic_stepsize} and adaptively chosen over-relaxed parameters compared with  the dynamic CQ algorithm \eqref{CQ}-\eqref{dynamic_stepsize}.

\vskip 1mm
The organization of the paper is as follows. In the next section we list several concepts and lemmas in a real Hilbert space that will be used in the convergence analysis of the algorithms. The main achievements of this paper are presented in Section 3.
In Section 4, numerical results are provided, which show the advantage of the proposed algorithm. A brief summary and commentary on the research results of this paper are  provided in Section 5.

\vskip 2mm

\section{Preliminaries}

Let $\mathcal{H}$ be a real Hilbert space with inner product $\langle \cdot,\cdot\rangle$ and norm $\|\cdot\|$. Some concepts and tools are listed in this section, which will be used in the proofs of our main results. The following notations will be used
throughout the rest of the paper.
\begin{itemize}
  \item [(i)] $I: \mathcal{H}\rightarrow \mathcal{H}$ denotes the identity operator.
  \item [(ii)]  $\rightarrow $ and $\rightharpoonup$ denote the strong convergence and  the weak convergence, respectively.
  \item [(iii)] $\omega_w(x_n) :=\{x\mid\exists \,\, {\rm subsequence}\,\,   \{x_{n_j}\}\subset\{x_n\}$
  such that $ x_{n_j} \rightharpoonup x\}$ denotes the $\omega$-weak limit point set of $ \{x_n\}$.
 \end{itemize}

The following identity will be used  for the main results (see  \cite[Corollary 2.15]{Bauschke-Combettes2017}):
\begin{equation}
\label{IDE}
  \|\alpha x+(1-\alpha)y\|^2=\alpha\|x\|^2+(1-\alpha)\|y\|^2-\alpha(1-\alpha)\|x-y\|^2,
\end{equation}
for all $\alpha \in \mathbb{R}$ and $(x,y) \in \mathcal{H}\times \mathcal{H}$.

We denote the interior and relative interior  of $K$ by $\Int K$ and $\ri K$, respectively. For $x \in \mathcal{H}$ and $r > 0$, we use
$B(x, r)$, and $\overline{B(x,r)}$ to denote the open metric ball and the closed metric ball at $x$ with radius $r$, respectively, that is
$$
B(x,r)=\{ y\in \mathcal{H} : \| x-y\|<r\} \quad\hbox{and}\quad \overline {B(x,r)}=\{ y\in \mathcal{H} : \| x-y\|\le r\}.
$$

{For a point $x$, the
the classical metric projection of
$x$ onto $K$ and the distance of $x$ from $K$, denoted by $P_K(x)$ and ${\rm d}_K(x)$ , are respectively defined by
$$
P_K(x):= \argmin\{  \| x - y\| :\,\,\forall y\in K\},
$$
and
$$
{\rm d}_K(x) := \inf\{\|x - y\| : \,\,\forall y\in K\}.
$$
}

Recall that a mapping $T: \mathcal{H}\rightarrow \mathcal{H}$ is called nonexpansive if
$$
\|Tx-Ty\|\leq \|x-y\|.
$$
It is well known that $P_K$ and $I-P_K$ are nonexpansive.

\begin{definition}
\label{Reg-Def}
\rm
The SFP \eqref{SFP} is said to satisfy the \texttt{bounded linear regularity property} if, for any $r>0$ such that $\Omega\cap \overline{B(0,r)}\neq \emptyset$, there exists $\gamma_r>0$ such that
$$
\gamma_r {\rm d}_\Omega(x) \leq  {\rm d}_Q(Ax)~~{\rm for~any }~x \in C\cap \overline{B(0,r)}.
$$
\end{definition}

Some sufficient conditions ensuring the regularity property  for the SFP (\ref{SFP}) are provided in
the following Lemma (see \cite{WangHu} for details).
\begin{lemma}
\label{lem-Reg-prop}
The SFP \eqref{SFP} satisfies the bounded linear regularity property provided
one of the following conditions holds:
\begin{itemize}
\item[{\rm(i)}]\,
$C$ and $Q$ are are polyhedrons;
\item[{\rm(ii)}]\,
$AC \cap  \Int Q\neq\emptyset$;
\item[{\rm(iii)}]\,
$A(\ri C) \cap Q\neq\emptyset$ and $Q$ is a polyhedron;
\item[{\rm(iv)}]\,
$AC\cap \ri Q \neq\emptyset$ and $C$ is a polyhedron and $Q$ is finite-codimensional;
\item[{\rm(v)}]\,
$A(\ri C)\cap\ri Q \neq\emptyset$ and  $Q$ is finite-codimensional.
\end{itemize}
\end{lemma}

\begin{lemma} {\rm\cite[Proposition 3.2]{Xu-SFP}}\label{fixed-point}
Let $\tau > 0$ and $x^* \in \mathcal{H}$. Then $x^*$ solves the SFP (\ref{SFP}) if and only if $x^*$ solves the fixed point equation
$$
x^* = P_C\big(x^*- \tau A^*(I - P_Q)Ax^*\big).
$$
\end{lemma}

	\begin{lemma}\label{convergence_lemma}
		\textnormal{(\cite[Lemma 2.47]{Bauschke-Combettes2017})} Let $K$ be a nonempty subset of $\mathcal{H}$, and let $\{x_n\}$ be a sequence in $\mathcal{H}$. Assume that the following conditions hold:
		\begin{itemize}
			\item[\rm(i)] for every $x\in K$, $\lim_{n\to\infty}\|x_n-x\|$ exists;
			\item[\rm(ii)] every weak sequential cluster point of $\{x_n\}$ belongs to $K$, i.e., $\omega_w(x_n)\subseteq K$.
		\end{itemize}
		Then the sequence $\{x_n\}$ converges weakly to a point in $K$.
	\end{lemma}

\begin{lemma}
\label{simple_inequality}
Let $N\ge0$, $D>0$, and $\lambda=\sigma(N+D)/D$ with $\sigma\in(0,1)$. Then
\[
\lambda\big[N+(1-\lambda)D\big]=\sigma(1-\sigma)\,\frac{(N+D)^2}{D}.
\]
\end{lemma}

\section{Main results}
In this section, we first introduce an adaptive relaxed CQ algorithm with dynamic stepsizes. Then, we prove the weak  and linear convergence of the proposed algorithm.

\subsection{Algorithm}

Now we present an adaptive strategy for selecting $\lambda_n$ for the scheme \eqref{relaxed-CQ}, which leads to the formulation of our proposed algorithm.

\begin{algorithm}[H]
\counterwithin{algorithm}{section}
\caption{Adaptive relaxed CQ algorithm with dynamic stepsizes}
\label{relaxed_Dynamic_CQ}
\begin{enumerate}
[label=Step \arabic*., leftmargin=*, align=left, itemsep=0pt, parsep=0pt, start=0, topsep=0pt,partopsep=0pt]
\item  Choose an initial guess $x_0\in\mathcal H$ arbitrarily. Choose $\rho\in(0,4)$ and $\sigma\in(0,1).$

\item   Compute
\begin{equation}
 \label{relaxed_CQ}
x_{n+1}=(1-\lambda_n)x_n+\lambda_nP_C\big(x_n - \tau_n A^*(I - P_Q)A x_n\big),
\end{equation}
where the stepsize $\tau_n$ is chosen in such a way that
\begin{equation}\label{tau_n}
\tau_n = \rho\frac{ f(x_n)}{\|\nabla f(x_n)\|^2}
\end{equation}
 and the relaxed parameter $\lambda_n$ is chosen in such a way that
\begin{equation}\label{lambda_n}
\lambda_n=\sigma\bigg(1+\rho(4-\rho)\frac{f^2(x_n)}{\|\nabla f(x_n)\|^2\|r_n\|^2}\bigg)
\end{equation}
with  $r_n=x_n-P_C\big(x_n - \tau_n A^*(I - P_Q)A x_n\big)$.

\item   If $\nabla f(x_n)=0$ or $r_n=0$, then  $x_{n+1}=x_n$ is a solution of the SFP \eqref{SFP} and the iterative process stops; otherwise, we set $n:=n+1$ and go to Step 1. 
\end{enumerate}
\end{algorithm}

\begin{remark}\rm
(i)
In Algorithm \ref{relaxed_Dynamic_CQ}, we can set $\tau_n=1$ and $\lambda_n=1$ when $\nabla f(x_n)=0$ or $r_n=0$. 
By Lemma \ref{fixed-point},   we see that if Algorithm \ref{relaxed_Dynamic_CQ} terminates in a finite (say $n$) step
of iterations, then $x_n$ is a solution of the SFP (\ref{SFP}). So in the rest of this section, we assume that
the Algorithm \ref{relaxed_Dynamic_CQ} does not terminate in any finite iterations, and hence generates an infinite
sequence $\{x_n\}$. Note that in this case, this sequence is well defined since $ \nabla f(x_n ) \neq 0$ and  $r_n\neq 0$ for any $n\ge 0$.

(ii) In Algorithm \ref{relaxed_Dynamic_CQ}, the parameters $\sigma$ and $\rho$ can vary with $n$. The convergence results remain valid under the  assumptions similar to those in \cite{Xu-dynamic-CQ}.

(iii) Algorithm \ref{relaxed_Dynamic_CQ} is an overrelaxation of the dynamic CQ method \eqref{CQ}-\eqref{dynamic_stepsize} when $\sigma$ is sufficiently close to one, since $\rho(4-\rho)\frac{f^2(x_n)}{\|\nabla f(x_n)\|^2\|r_n\|^2}>0.$

(iv) In Algorithm \ref{relaxed_Dynamic_CQ}, we assume that the projections \(P_C\) and \(P_Q\) can be easily computed. In some practical scenarios, however, evaluating these projections may be infeasible or computationally expensive. To address this issue, researchers have proposed relaxed variants by replacing the projections onto $C$ and $Q$ with projections onto half-spaces containing $C$ and $Q$ \cite{Dong-general-splitting-method,Dong-Optimal,{Yangqz}}. A relaxed variant of Algorithm \ref{relaxed_Dynamic_CQ} can be constructed in a similar manner.
\end{remark}

\subsection{Weak convergence}

From the proof of Lemma 3.4 in \cite{Xu-dynamic-CQ}, it follows the contraction inequality of the dynamic CQ algorithm \eqref{CQ}-\eqref{dynamic_stepsize}  given by 
\begin{equation}
\label{Contraction-dynamic-CQ}
\|x_{n+1}-z\|^2 \le  \|x_n - z\|^2 - \rho_n(4-\rho_n)\frac{f^2(x_n)}{\|\nabla f(x_n)\|^2},
\end{equation}
where $z\in \Omega$ is arbitrary.

By \eqref{Contraction-dynamic-CQ},  we have the below result.
\begin{lemma}\label{lem_dynamic_CQ}
Let 
$
y_n=P_C\big(x_n - \tau_n A^*(I - P_Q)A x_n\big),
$
where $\tau_n = \rho\frac{ f(x_n)}{\|\nabla f(x_n)\|^2}.$
 Then, for any $z\in\Omega$,
$$
\|y_n-z\|^2 \le  \|x_n - z\|^2 - \rho(4-\rho)\frac{f^2(x_n)}{\|\nabla f(x_n)\|^2}. 
$$
\end{lemma}

The following lemma establishes a contraction inequality of Algorithm \ref{relaxed_Dynamic_CQ} together with several related results.
\begin{lemma}\label{lem:3.2}
 Let the sequence $\{x_n\}$ be generated by Algorithm \ref{relaxed_Dynamic_CQ} and  $y_n=P_C\big(x_n - \tau_n A^*(I - P_Q)A x_n\big)$.
Let $D_n=\|x_n-y_n\|^2,
$ and $
N_n=\rho(4-\rho)\frac{f^2(x_n)}{\|\nabla f(x_n)\|^2}
$. Then
\begin{itemize}
\item[\rm(i)] for any $z\in\Omega$,
\begin{equation}
\label{Contraction-rCQ}
\|x_{n+1}-z\|^2\le\|x_n-z\|^2-\sigma(1-\sigma)\,\frac{(N_n+D_n)^2}{D_n}.
\end{equation}
    \item[\rm(ii)] for every $z\in \Omega$, $\lim_{n\to\infty}\|x_n-z\|$ exists.
    \item[\rm(iii)] $\{x_n\}$ is a bounded sequence.
    \item[\rm(iv)] $\sum_{n=0}^\infty\frac{(N_n+D_n)^2}{D_n}
< \infty$.
\end{itemize}
\end{lemma}

\begin{proof}
Since $x_{n+1}=(1-\lambda_n)x_n+\lambda_ny_n$, it follows from \eqref{IDE} and Lemma \ref{lem_dynamic_CQ} that
$$
\begin{aligned}
\|x_{n+1}-z\|^2
=&(1-\lambda_n)\|x_n-z\|^2+\lambda_n\|y_n-z\|^2
-\lambda_n(1-\lambda_n)\|x_n-y_n\|^2\\
\le&(1-\lambda_n)\|x_n-z\|^2+\lambda_n\bigg(\|x_n - z\|^2 - \rho(4-\rho)\frac{f^2(x_n)}{\|\nabla f(x_n)\|^2}\bigg)\\
&
-\lambda_n(1-\lambda_n)\|x_n-y_n\|^2\\
=&\|x_n-z\|^2-\lambda_n\big(N_n
+(1-\lambda_n)D_n\big).
\end{aligned}
$$
 Since $\lambda_n=\sigma(N_n+D_n)/{D_n},$
applying Lemma \ref{simple_inequality}, we have
$$
\|x_{n+1}-z\|^2\le\|x_n-z\|^2-\sigma(1-\sigma)\,\frac{(N_n+D_n)^2}{D_n},
$$
which in turns implies immediately the conclusions (ii)-(iv).
\end{proof}

The following lemma characterizes the range of $\sigma$, for which Algorithm \ref{relaxed_Dynamic_CQ} enjoys a stronger contraction inequality than the dynamic CQ algorithm \eqref{CQ}-\eqref{dynamic_stepsize}.
\begin{lemma}
\label{comparison}
Let
$
D_n=\|x_n-y_n\|^2
$ 
and
$
N_n=\rho(4-\rho)\frac{f^2(x_n)}{\|\nabla f(x_n)\|^2}.
$
 Then
\begin{equation}
\label{equ00}
\sigma(1-\sigma)\frac{(N_n+D_n)^2}{D_n}\geq N_n
\end{equation}
if and only if
\begin{equation}
\label{equ1}
\frac{\min\{N_n,D_n\}}{N_n+D_n}
\leq \sigma \leq
\frac{\max\{N_n,D_n\}}{N_n+D_n}.
\end{equation}
Moreover, if $N_n\neq D_n$, then
the inequality \eqref{equ00} is strict exactly
if and only if
$
\frac{\min\{N_n,D_n\}}{N_n+D_n}
<
\sigma
<
\frac{\max\{N_n,D_n\}}{N_n+D_n}.
$
\end{lemma}

\begin{proof}
It is easy to verify that
\[
\begin{aligned}
\sigma(1-\sigma)\frac{(N_n+D_n)^2}{D_n}-N_n
&=
\frac{\sigma(1-\sigma)(N_n+D_n)^2-N_nD_n}{D_n}\\
&=
\frac{\bigl[\sigma(N_n+D_n)-N_n\bigr]
\bigl[D_n-\sigma(N_n+D_n)\bigr]}{D_n}.
\end{aligned}
\]
Therefore,
\eqref{equ00} holds
if and only if
\begin{equation}
\label{equ2}
\bigl[\sigma(N_n+D_n)-N_n\bigr]
\bigl[\sigma(N_n+D_n)-D_n\bigr]\leq0,
\end{equation}
which is equal to \eqref{equ1}.
If $N_n\neq D_n$, the inequality \eqref{equ2} is strict exactly when $
\frac{\min\{N_n,D_n\}}{N_n+D_n}
<
\sigma
<
\frac{\max\{N_n,D_n\}}{N_n+D_n}
$, and so is \eqref{equ00}.
\end{proof}

\begin{remark}\rm
By Lemma \ref{comparison},  Algorithm \ref{relaxed_Dynamic_CQ}
has a stronger contraction property than the dynamic CQ algorithm \eqref{CQ}-\eqref{dynamic_stepsize} when $D_n\neq N_n$ and $
\frac{\min\{N_n,D_n\}}{N_n+D_n}
< \sigma <
\frac{\max\{N_n,D_n\}}{N_n+D_n}.
$

Since
$
\frac{\min\{N_n,D_n\}}{N_n+D_n}
=
\frac12-\frac{|N_n-D_n|}{2(N_n+D_n)}
$
and
$
\frac{\max\{N_n,D_n\}}{N_n+D_n}
=
\frac12+\frac{|N_n-D_n|}{2(N_n+D_n)},
$
\eqref{equ1} can be equivalently written as
\[
\left|\sigma-\frac12\right|
\leq
\frac{|N_n-D_n|}{2(N_n+D_n)}.
\]
The above inequality always holds when $\sigma=\frac12$.
Hence, in this case,  the descent term in \eqref{Contraction-rCQ} of Algorithm \ref{relaxed_Dynamic_CQ}
 is always no smaller than that in \eqref{Contraction-dynamic-CQ} of the dynamic CQ algorithm \eqref{CQ}-\eqref{dynamic_stepsize}.
\end{remark}

\begin{theorem}
\label{weak-convergence-Theorem}
Let $\{x_n\}$ be generated by Algorithm \ref{relaxed_Dynamic_CQ}. Then $\{x_n\}$ converges weakly to a point in $\Omega$.
\end{theorem}

\begin{proof}
By Lemma \ref{lem:3.2}(iv), we have
$\bigg(\|x_n-y_n\|^2+\rho(4-\rho)\frac{f^2(x_n)}{\|\nabla f(x_n)\|^2}\bigg)^2/\|x_n-y_n\|^2\rightarrow0$
as $n\rightarrow\infty.$ Using Lemma \ref{lem:3.2}(iii) and Lemma \ref{lem_dynamic_CQ}, we have the boundedness of $\{x_n\}$ and $\{y_n\}$. Therefore,
$$
\bigg(\|x_n-y_n\|^2+\rho(4-\rho)\frac{f^2(x_n)}{\|\nabla f(x_n)\|^2}\bigg)^2
=\frac{\bigg(\|x_n-y_n\|^2+\rho(4-\rho)\frac{f^2(x_n)}{\|\nabla f(x_n)\|^2}\bigg)^2}{\|x_n-y_n\|^2}\cdot\|x_n-y_n\|^2
\rightarrow0,
$$
as $n\rightarrow\infty.$
And we conclude
\begin{equation}
\label{eq1}
\|x_n-y_n\|\rightarrow0,\quad\hbox{and}\quad \frac{f^2(x_n)}{\|\nabla f(x_n)\|^2}\rightarrow0,\quad \hbox{as}\,\,n\rightarrow\infty.
\end{equation}
Since $\nabla f(x)=A^*(I-P_Q)Ax$, we have $\|\nabla f(x_n)\|\le\|A\|\,\|(I-P_Q)Ax_n\|=\|A\|\sqrt{2f(x_n)}$, so
\[
\frac{f^2(x_n)}{\|\nabla f(x_n)\|^2}\ge\frac{f(x_n)}{2\|A\|^2}\ \ge 0.
\]
Combined with $f^2(x_n)/\|\nabla f(x_n)\|^2\to0$ in \eqref{eq1}, this yields
\begin{equation}
\label{eq2}
f(x_n)\rightarrow0,\qquad\text{i.e.,}\qquad \|(I-P_Q)Ax_n\|\rightarrow0,\quad \hbox{as}\,\,n\rightarrow\infty.
\end{equation}

Since $\{x_n\}$ is bounded from Lemma \ref{lem:3.2}(iii), we have $\omega_w(x_n)\neq\emptyset$. Take arbitrarily $\bar x\in \omega_w(x_n)$ and there exists a weakly convergent subsequence $x_{n_j}\rightharpoonup\bar x$ as $j\rightarrow\infty.$ By \eqref{eq1}, we also have $y_{n_j}\rightharpoonup\bar x$ as $j\rightarrow\infty.$
From the fact that $y_{n_j}\in C$ and $C$ is closed and convex,  we get $\bar x\in C$.
 Since $A$ is a bounded linear operator,  $Ax_{n_j}\rightharpoonup A\bar x$ as $j\rightarrow\infty.$ Using \eqref{eq2} and the demiclosedness of $I-P_Q$ at $0$ (see, e.g., \cite{GK1990}), we get $A\bar x\in Q$.
Therefore $\bar x\in\Omega$.
Combining Lemmas \ref{convergence_lemma} and \ref{lem:3.2}(i), we obtain the desired result.
\end{proof}

\subsection{Linear convergence}

In this subsection, we show the linear convergence of Algorithm \ref{relaxed_Dynamic_CQ} provided that the SFP \eqref{SFP} satisfies the bounded linear regularity property.
\begin{theorem}
\label{linear-convergence}
Suppose  that the SFP \eqref{SFP} satisfies the bounded linear regularity property.
Then  the sequence $\{x_n\}$ generated by Algorithm \ref{relaxed_Dynamic_CQ} converges linearly to a point $x^*\in\Omega$ satisfying
\begin{equation}
\label{eq0}
\|x_n-x^*\|\ \le cq^n,\qquad\forall n\ge0,
\end{equation}
where $c\in (0,\infty)$ and $q \in (0,1)$.
\end{theorem}

\begin{proof}
Since the sequence $\{y_n\}$ is bounded from Lemma \ref{lem_dynamic_CQ} and \ref{lem:3.2}(iii), there exists $r$ such that $\{y_n\}\subseteq C\cap \overline{{B}(0,r)}$.
By assumption that the SFP \eqref{SFP} satisfies the bounded linear regularity property, it follows from definition \ref{Reg-Def} that there exists $\gamma_r>0$ such that
\begin{equation}
\label{eq3}
\mathrm{d}_Q(Ay_n)\ge \gamma_r \mathrm{d}_\Omega(y_n) \quad \text{for any } n\ge 0.
\end{equation}
Let $D_n=\|x_n-y_n\|^2$. Since $I-P_Q$ is nonexpansive, we have
\begin{equation}
\label{eq4}
\begin{aligned}
\mathrm d_Q(Ay_n)&=\|(I-P_Q)Ay_n\|\\
& \le\ \|(I-P_Q)Ax_n\|+\|A(x_n-y_n)\|\\
& \le\ \sqrt{2f(x_n)}+\|A\|\sqrt{D_n}.
\end{aligned}
\end{equation}
Combining \eqref{eq3} and \eqref{eq4}, we get
$$
\mathrm d_\Omega(y_n)\ \le\ \frac{1}{\gamma_r}\sqrt{2f(x_n)}+\frac{\|A\|}{\gamma_r}\sqrt{D_n}. 
$$
Since  the triangle inequality implies
$
\mathrm d_\Omega(x_n) \le\mathrm d_\Omega(y_n)+\|x_n-y_n\| = \mathrm d_\Omega(y_n)+\sqrt{D_n},
$
we have
\[
\mathrm d^2_\Omega(x_n)\ \le\ \frac{4}{\gamma_r^2}f(x_n)+2\Big(\frac{\|A\|}{\gamma_r}+1\Big)^2D_n.
\]
Let $
K=\max\left\{\frac{2}{\gamma_r^2},\ 2\Big(\frac{\|A\|}{\gamma_r}+1\Big)^2\right\},
$
then
\begin{equation}
\label{eq7}
\mathrm d^2_\Omega(x_n)\ \le\ K\big(D_n+2f(x_n)\big).
\end{equation}
Let $
N_n=\rho(4-\rho)\frac{f^2(x_n)}{\|\nabla f(x_n)\|^2}
$. Since $\|\nabla f(x_n)\|\le\|A\|\sqrt{2f(x_n)}$, we have
$
N_n\ge\frac{\rho(4-\rho)}{2\|A\|^2}f(x_n).
$
Let $\eta=\min\left\{1,\ \frac{\rho(4-\rho)}{2\|A\|^2}\right\}\in(0,1]$. Then $N_n\ge \eta\,f(x_n)$ and $D_n\ge \eta\,D_n$, so
\begin{equation}
\label{eq10}
\frac{(N_n+D_n)^2}{D_n}\ \ge\ D_n+2N_n\ \ge\ \eta\big(D_n+2f(x_n)\big)\ \overset{\eqref{eq7}}{\ge}\ \frac{\eta}{K}\,\mathrm d^2_\Omega(x_n).
\end{equation}
Using Lemma \ref{lem:3.2}(i)
 and  \eqref{eq10}, we obtain
$$
\begin{aligned}
\mathrm d^2_\Omega(x_{n+1})\le&\ \mathrm  d^2_\Omega(x_n)-\sigma(1-\sigma)\,\frac{(N_n+D_n)^2}{D_n}\\
\ \le& \Big(1-\frac{\sigma(1-\sigma)\eta}{K}\Big)\mathrm d^2_\Omega(x_n),
\end{aligned}
$$
which yields
\begin{equation}
\label{th51-3}
{\rm d}_\Omega(x_{n+1})
\leq q^{n-N+1}{\rm d}_\Omega(x_N),\quad{\rm for ~any}~n\geq N.
\end{equation}
where $q:=\sqrt{1-\dfrac{\sigma(1-\sigma)\eta}{K}}\in(0,1)$.
Now fix $n\geq N$. Recalling that $\{\|x_m-P_\Gamma(x_{n+1})\|\}_{m\geq n}$ is monotone decreasing by Lemma \ref{lem:3.2}(i), we conclude, for any $m>n$, that
$$
\aligned
\|x_m-x_{n+1}\|&\leq\|x_m-P_\Gamma(x_{n+1})\|+\|x_{n+1}-P_\Gamma(x_{n+1})\|\\
&\leq2\|x_{n+1}-P_\Gamma(x_{n+1})\|=2{\rm d}_\Omega(x_{n+1}).
\endaligned
$$
This, together with \eqref{th51-3}, implies that
\begin{equation}
\label{equation1}
\|x_m-x_{n+1}\|\leq2q^{n-N+1}{\rm d}_\Omega(x_N),
\quad{\rm for ~any}~m>n\geq N.
\end{equation}
It follows that $\{x_n\}$ is a Cauchy sequence and converges to some point
$x^*$. By Theorem \ref{weak-convergence-Theorem}, $x^*$ is a solution of the
SFP \eqref{SFP}.
 Using \eqref{equation1} again, we get
$$
\|x_{n+1}-x^*\|
\leq2q^{n-N+1}{\rm d}_\Omega(x_N)
\quad {\rm for ~any}~n\geq N.
$$
Then \eqref{eq0} holds, where 
$$
c:=\max\left\{2{\rm d}_\Omega(x_N)q^{-N},
\max_{i=1,\ldots,N}\|x_i-x^*\|q^{-i}\right\}>0.
$$
Thus we complete the proof.
\end{proof}

\section{Numerical results}

In this section, we apply the Algorithm \ref{relaxed_Dynamic_CQ}  to solve the LASSO problem. First, we investigate the optimal selections of $\sigma$ and  the corresponding values of $\lambda_n$ for different $\sigma$. Then we compare  Algorithm \ref{relaxed_Dynamic_CQ} with the dynamic CQ method \eqref{CQ}-\eqref{dynamic_stepsize}.

\begin{example}
\label{Lasso_Problem} {\rm Consider the following LASSO problem \cite{Tibshirani1996}:
\begin{equation}\label{lasso}
\min\left\{\frac{1}{2}\|Ax-b\|_{2}^{2} \,:\, x\in  \Bbb R^t, \|x\|_1 \leq \delta\right\},
\end{equation}
where $A\in  \Bbb R^{s\times t}$, $s< t$, $b\in  \Bbb R^s$ and $\delta >0$. We generate the system matrix $A$ from a standard normal distribution with mean zero and unit variance. The true sparse signal $x^*$ is generated from uniformly distribution in the interval $[-2,2]$ with random $K$ position nonzero while the rest is kept zero. The sample data $b= Ax^*$. Under certain conditions on matrix $A$, the solution of the minimization problem (\ref{lasso}) is
equivalent to the $\ell_0$-norm solution of the underdetermined linear system.}
\end{example} 

It is easy to see that the LASSO problem \eqref{lasso} can be cast as the  SFP (\ref{SFP}) with $C = \{ x \,|\, \|x\|_1 \leq \tau\}$ and $Q = \{b\}$. Since the projection onto the closed convex set $C$ does not admit a closed-form expression, we employ the subgradient projection instead; see, e.g., \cite{Dong-general-splitting-method,Dong-Optimal}.
We initialize the algorithms at the
origin and terminate it when
$$
\frac{\|x_n-x_{n-1}\|}{\max\{1,\|x_n\|\}}<10^{-4}.
$$
We report the final error between the reconstructed signal
and the true signal, that is, $Err = \|x_n - x^*\|$. In the numerical
results listed in the following table and figure, ``Iter"  denotes the number of iterations.

\begin{figure}[htbp]
\setlength{\floatsep}{0pt} \setlength{\abovecaptionskip}{0pt}
\centering
\scalebox{0.5} {\includegraphics{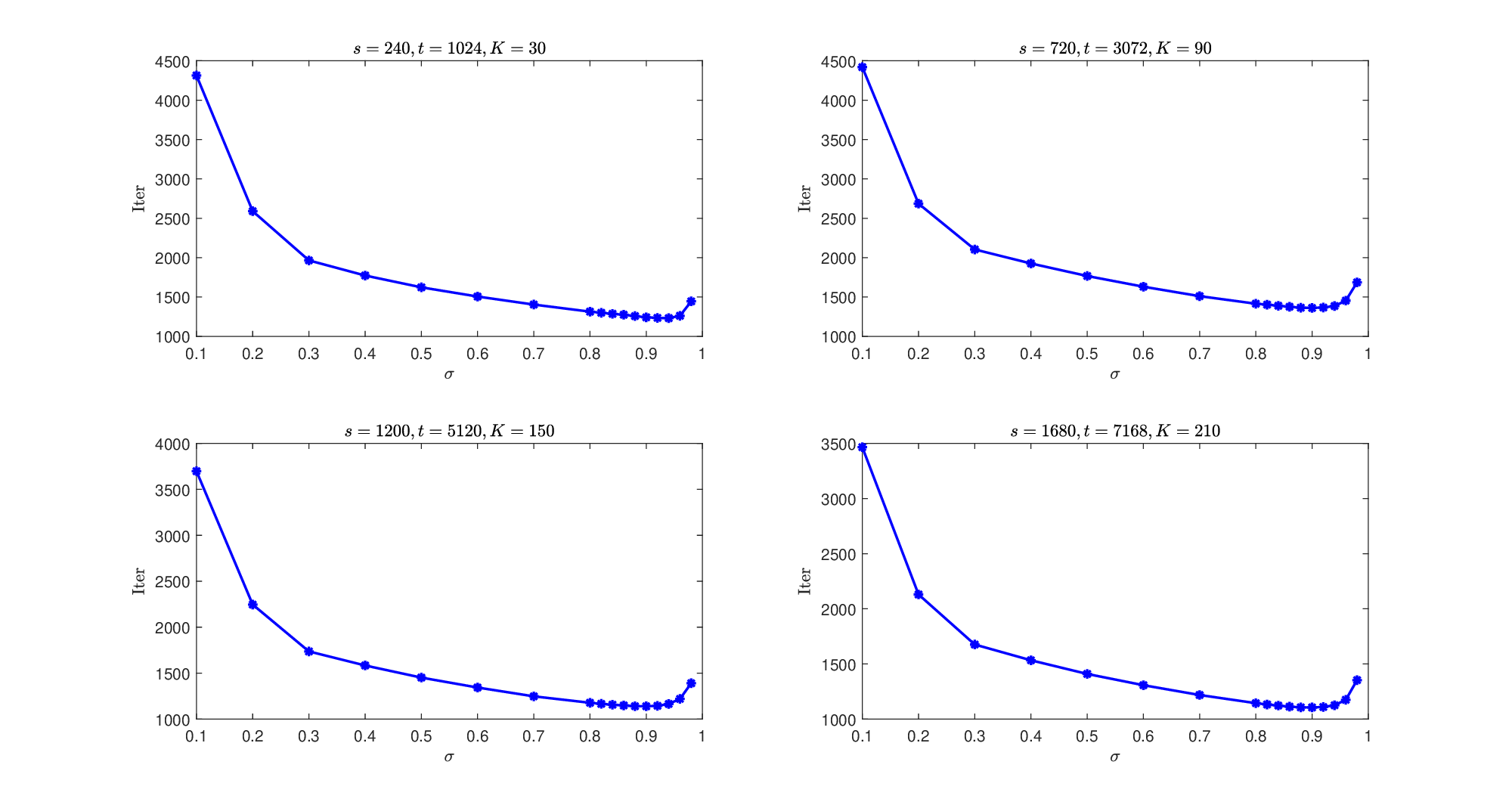}}
\caption{Comparison of the iteration numbers for different $\sigma$ for Example  \ref{Lasso_Problem}}
\label{fig-sigma}
\end{figure}

\begin{figure}[htbp]
\setlength{\floatsep}{0pt} \setlength{\abovecaptionskip}{0pt}
\centering
\scalebox{0.5} {\includegraphics{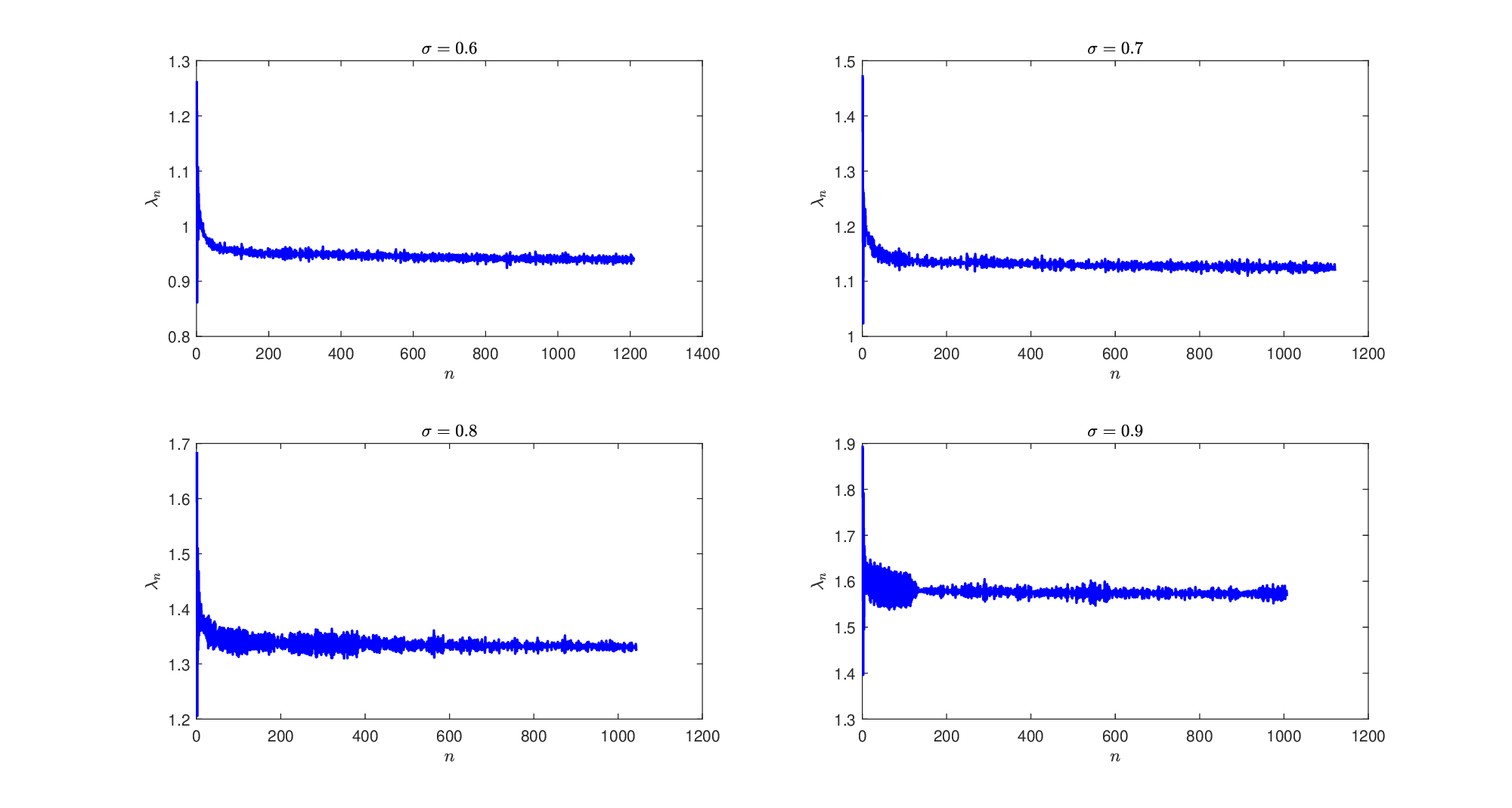}}
\caption{$\lambda_n$ given in \eqref{lambda_n} versus the iteration numbers for different $\sigma$ for  Example \ref{Lasso_Problem} with  $(s, t,K)=(1680,7168,210)$}
\label{fig-lambda}
\end{figure}

To obtain the optimal value of $\sigma\in(0,1)$, we tested  four problem sizes $(s,t,K)=(240i, 1048i,30i)$ for $i = 1, 3,5,7$ and compared the corresponding iteration numbers. From Figure \ref{fig-sigma}, it can be seen that the optimal choice of $\sigma$ is around $0.9$. Therefore, we took $\sigma=0.9$ in the sequent  numerical experiment. As shown in Figure \ref{fig-lambda}, $\lambda_n$  oscillates with $n$ and its average value increases with $\sigma$. Moreover, $\lambda_n>1$ for $\sigma=0.7, 0.8$ and $0.9$. Hence, for these choices of $\sigma$, Algorithm \ref{relaxed_Dynamic_CQ} is an overrelaxation of the dynamic CQ method \eqref{CQ}-\eqref{dynamic_stepsize}.

\begin{table}[htbp]
\label{table}
\footnotesize
\centering
\caption{Computational results for solving the SFP \eqref{SFP} with Algorithm \ref{relaxed_Dynamic_CQ} and the dynamic CQ algorithm}
\vskip 2mm
\begin{tabular}{ccccccccccccccccc}
\hline
\multicolumn{3}{c}{Problem size}  & \multicolumn{3}{c}{Iter} &
\multicolumn{3}{c}{CPU time}&\multicolumn{3}{c}{Err} \\
 \cline{1-3} \cline{5-6} \cline{8-9} \cline{11-12}
$s$ & $t$ & $K$ && Alg. \ref{relaxed_Dynamic_CQ}&dynamic CQ && Alg. \ref{relaxed_Dynamic_CQ}&dynamic CQ&& Alg. \ref{relaxed_Dynamic_CQ}&dynamic CQ \\
\hline
 \hline
240&1024&30&&1127&4275.3&&1.5188&5.2969&&0.0025&0.0028\\
480&2048&60&& 1146.5&4283.7&&4.5047&14.8219&&0.0035&0.0041\\
720&3072&90&&1155.7&4299.4&&17.8062& 61.9594&&0.0043&0.0050\\
960&4096&120&&1041.2&3902.4&&27.7859&110.0922&&0.0047&0.0054\\
1200&5120&150&&1146.7&4273.9&&55.7578&197.8453&&0.0057&0.0065\\
1440&6144&180&&1129.3&4168.4&&69.5734&256.9828&&0.0062&0.0071\\
1680&7168&210&&1.0659&3.9873&&85.3063&319.9078&&0.0065&0.0075\\
1920&8192&240&&1107&4147.9&&129.4328&475.2547&&0.0069&0.0079\\
 \hline
\hline
\end{tabular}
\label{results2}
\end{table}

\begin{figure}[htbp]
\setlength{\floatsep}{0pt} \setlength{\abovecaptionskip}{0pt}
\centering
\scalebox{0.5} {\includegraphics{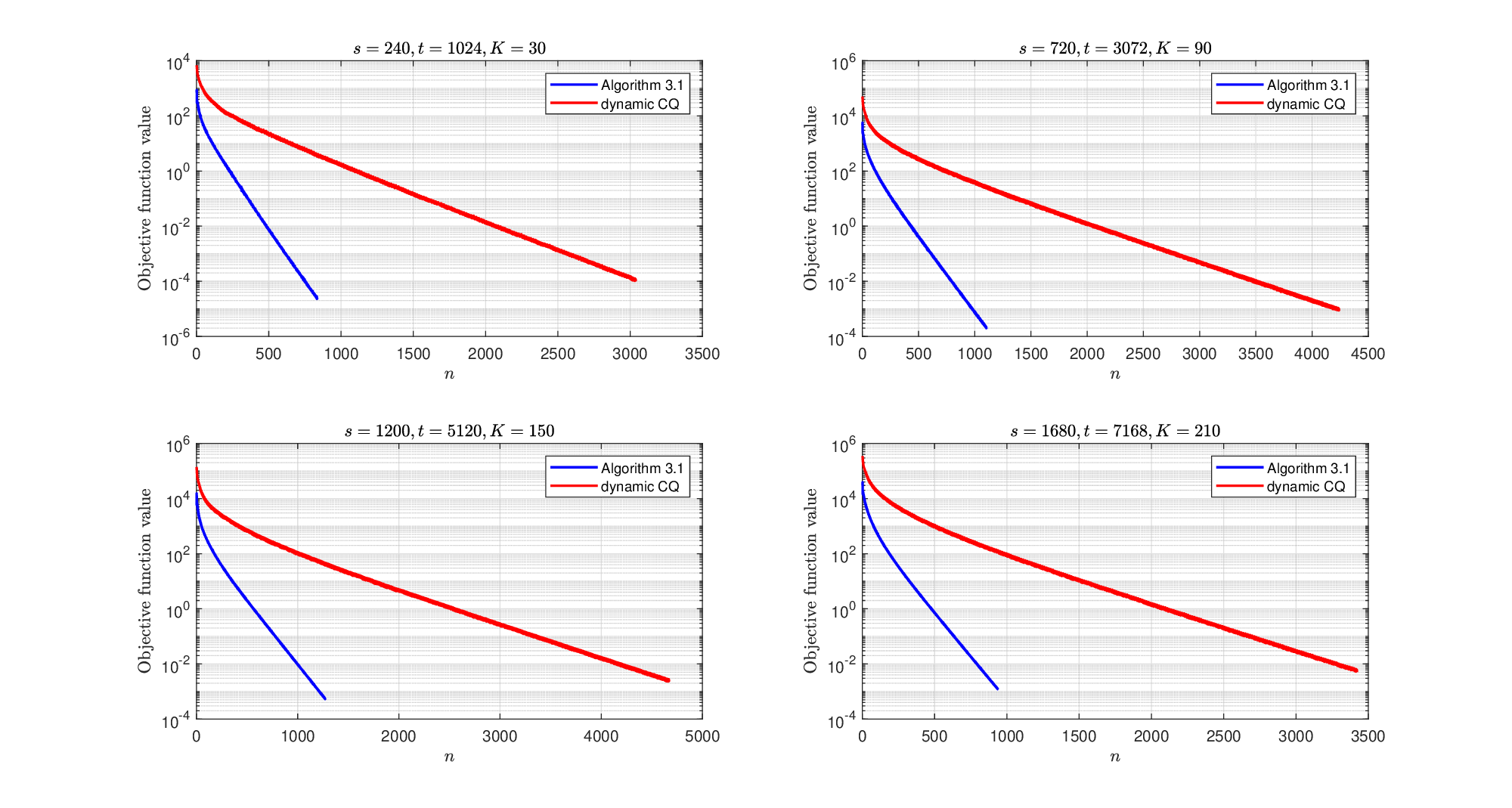}}
\caption{The objective function value  versus the iteration numbers for Example \ref{Lasso_Problem}}
\label{fig2}
\end{figure}

For comparing Algorithm \ref{relaxed_Dynamic_CQ} with the dynamic CQ method, we considered $(s,t,K) = (240i, 1048i,30i)$ for $i = 1, 2,\ldots,8$.
The corresponding results are reported in Table \ref{results2}. It can be seen that Algorithm \ref{relaxed_Dynamic_CQ} requires less than one-third of the number of iterations and CPU time of the dynamic CQ method. Figure \ref{fig2} shows the  objective function value versus iteration numbers when $(s,t,K)=(240i, 1048i,30i)$ for $i = 1, 3,5,7$.
We can see from Table \ref{results2} and Figure \ref{fig2} that Algorithm \ref{relaxed_Dynamic_CQ}  performs better than the dynamic CQ method in terms of   the computational accuracy and the objective function value. 

\section{Some concluding remarks}

In this paper, we study a relaxed version of the dynamic CQ method by introducing an adaptive strategy for selecting the relaxation parameters. The weak and linear convergence the proposed algorithm is established under the standard assumptions. Numerical experiments demonstrate that the adaptive overrelaxation can accelerate the convergence.

It is of interest to further investigate the convergence of the CQ algorithm \eqref{CQ} with fixed stepsizes and Armijo-type line-search strategies. Another interesting direction is to develop relaxed variants of general splitting methods \cite{Dong-general-splitting-method}, primal-dual-based iterative methods \cite{Zhaojing} and Polyak's gradient method \cite{Wangfh}.

\section*{Acknowledgments}
This work is supported by Flight Training-Related ``Open Competition" Projects--Fundamental Research Funds for the Central Universities [grant number 3122026FX05],

\end{document}